\documentclass[11pt]{article}
\usepackage{fullpage, amsmath, amssymb, amsfonts, amscd, amsxtra, mathrsfs, stmaryrd, amsthm, url, hyperref}
\usepackage[all]{xy}

\newtheorem{theorem}{Theorem}[section]
\newtheorem{lemma}[theorem]{Lemma}
\newtheorem{cor}[theorem]{Corollary}

\newtheorem{prop}[theorem]{Proposition}
\theoremstyle{definition}
\newtheorem{defn}[theorem]{Definition}

\numberwithin{equation}{theorem}

\def\ZZ{{\mathbb Z}}

\def\Qp{{\mathbb{Q}}_p}

\def\MM{\mathfrak{m}}
\def\D{\mathrm{D}}
\def\ho{\widehat{\otimes}}

\def\Gal{\mathrm{Gal}}
\def\rig{\mathrm{rig}}

\def\dif{\mathrm{dif}}

\def\ho{\widehat{\otimes}}

\def\bcris{\mathbf{B}_{\mathrm{crys}}}

\def\bdR{\mathbf{B}_{\rm dR}}

\def\m{(\varphi,\Gamma)}
\def\dcris{\mathrm{D}_{\mathrm{crys}}}
\def\ddR{\mathrm{D}_{\mathrm{dR}}}

\def\ra{\rightarrow}

\newcommand{\bdag}[1]{\mathbf{B}^{\dagger #1}}

\newcommand{\brig}[2]{\mathbf{B}^{\dagger #1}_{\mathrm{rig} #2}}

\title{The Eigencurve is Proper}
\author{%
Hansheng Diao\\
       Department of Mathematics\\
       Harvard University\\
       hansheng@math.harvard.edu
\and
       Ruochuan Liu\\
     Beijing International Center\\
        for Mathematical Research\\
        Peking University\\
        liuruochuan@math.pku.edu.cn}
\date{}

\begin{document}
\maketitle
\begin{abstract}
We prove in this paper that for any prime $p$ and tame level $N$, the projection from the eigencurve to the weight space satisfies a rigid analytic version of the ``valuative criterion for properness'' introduced by Buzzard and Calegari. This gives a negative answer to a question of Coleman and Mazur.
\end{abstract}


\section{Introduction}
Let $p$ be a prime number. The purpose of this paper is to answer the following question raised by Coleman and Mazur in \cite{CM98}:
\begin{quote}
Do there exist $p$-adic analytic families of overconvergent eigenforms of finite slope parameterized by a punctured disc, and converging, at the puncture, to an overconvergent eigenform of infinite slope?
\end{quote}

\noindent In [\emph{loc. cit.}], for $p>2$, Coleman and Mazur constructed a $\Qp$-rigid analytic curve over the weight space whose $\mathbb{C}_p$-points parametrize all finite slope overconvergent $p$-adic eigenforms of tame level 1. They call it the $p$-adic \emph{eigencurve} of tame level 1. Such a construction was later generalized to all primes and tame levels by Buzzard \cite{Buz07}. In this framework, as suggested by Buzzard and Calegari \cite{BuCa06}, one can formulate the question of Coleman and Mazur as whether the projection from the eigencurve to the weight space satisfies the \emph{valuative criterion for properness}. 

In the past decade, some progress has been made towards this problem. In \cite{BuCa06}, the properness was proved for $p=2$ and tame level $N=1$. In \cite{Cal08}, the properness was proved at the weights $\chi^k\psi$ where $\chi$ is the $p$-adic cyclotomic character, $k\in\mathbb{Z}$ and $\psi$ is a finite order character of conductor dividing $N$. Both of the works are based on explicit estimates for the convergence of certain overconvergent modular forms. 

In this paper, we will show that for all primes $p$ and tame levels $N$, the answer to the question of properness is yes. That is, the answer to the question of Coleman and Mazur is no. More precisely, the main result of this paper is the following theorem.

\begin{theorem}\label{main}
Let $\mathcal{C}_{p,N}$ be the Coleman-Mazur eigencurve of tame level $N$, and let $\pi:\mathcal{C}_{p,N}\ra \mathcal{W}_N$ denote the natural projection to the weight space. Let $D$ be the closed unit disk over some finite extension $L$ over $\Qp$, and let $D^{\ast}$ be the punctured disk. Suppose $h: D^{\ast}\rightarrow \mathcal{C}_{p,N}$ is a morphism of rigid analytic spaces such that $\pi\circ h$ extends to $D$.\footnote{In fact,  any rigid analytic morphism $D^\ast\ra\mathcal{W}_N$ extends uniquely to a rigid analytic morphism $D\ra\mathcal{W}_N$ by Lemma \ref{lem:extension}; hence this condition always holds.}  Then $h$ extends to a morphism $\tilde{h}: D\rightarrow \mathcal{C}_{p, N}$ compatible with $\pi\circ h$.
\end{theorem}

We have to point out that although this property is named ``properness of the eigencurve'', the projection $\pi$ is actually not proper in the sense of rigid analytic geometry because it is of infinite degree.

In the rest of the introduction we will sketch the steps to prove Theorem \ref{main} and the structure of the paper. As will be clear to the reader, our method is totally different from the one of \cite{BuCa06} and \cite{Cal08}. In fact, by \cite{CM98} (for $p>2$ and $N=1$) and \cite{KL03} (for general $p$ and $N$), there exists a family of $p$-adic representations $V_{\tilde{\mathcal{C}}}$ of $G_{\mathbb{Q}}$ over the normalization $\tilde{\mathcal{C}}_{p,N}$ of $\mathcal{C}_{p,N}$ interpolating (the semi-simplifications of) the Galois representations associated to classical eigenforms. Granting this global input, our approach to Theorem \ref{main} is purely local and Galois theoretical. We make extensive use of the recent advances \cite{BeCo08}, \cite{KL10}, \cite{Bel13}, \cite{KPX}, and especially  \cite{Liu12} in $p$-adic Hodge theory in rigid analytic families. In \S3, we will give a brief review on the family version of the functors $\ddR^+$, $\dcris^+$, $\D^{\dagger}_{\rig}$ and $\D^+_{\dif}$ for the sake of the reader. 

By its construction, one may regard the eigencurve $\mathcal{C}_{p,N}$ as an analytic subspace of 
\[
X_p\times \mathbb{G}_m\times\prod_{\text{prime}\hspace{0.5mm}q |N} \mathbb{A}^1
\] 
where $X_p$ is the deformation space of the pseudo-representations associated to all isomorphism classes of $p$-modular, tame level $N$, residue representations of $G_{\mathbb{Q}}$. In \S2, we will show that the composition 
\[
D^\ast\ra\mathcal{C}_{p,N}\hookrightarrow X_p\times \mathbb{G}_m\times\prod_{\text{prime}\hspace{0.5mm}q |N} \mathbb{A}^1
\ra X_p
\] 
extends to a morphism of rigid analytic spaces on $D$. As a consequence, we obtain a family of pseudo-representations on $D$ by pulling back the universal pseudo-representation over $X_p$. Applying a result of \cite{CM98}, one can convert it to a family of $p$-adic representations $V_D$ over $D$. 

The difficult part is to show that the composition 
\[
D^\ast\ra\mathcal{C}_{p,N}\hookrightarrow X_p\times \mathbb{G}_m\times\prod_{\text{prime}\hspace{0.5mm}q |N} \mathbb{A}^1
\ra \mathbb{G}_m
\] 
extends to a morphism on $D$. This amounts to showing that the pullback of the $U_p$-eigenvalue, which is denoted by $\alpha$, is nonzero at the puncture. We achieve this by showing that the specialization of $V^\ast_D$, the dual of $V_D$, at the puncture has a nonzero crystalline period with Frobenius eigenvalue $\alpha(0)$. Noticing that the crystalline Frobenius is always injective, we conclude that $\alpha(0)$ is nonzero.

To this end, we compare the positive crystalline and de Rham periods of $V_D^\ast$. Recall that it was proved by Kisin that the positive crystalline and de Rham periods of $V_{\tilde{\mathcal{C}}}^\ast$ coincide on ``$Y$-small" affinoid subdomains of $\tilde{\mathcal{C}}_{p,N}$ \cite{Kis03}. This property was later strengthened to all affinoid subdomains by the work of the second author \cite{Liu12}. Using the results of \cite{Liu12}, we show in \S4 that positive crystalline and de Rham periods also coincide on the punctured disk. More precisely, for any affinoid subdomain $M(R)$ of $D^\ast$, we have that $\dcris^+(V_{R}^\ast)^{\varphi=\alpha}$ is a locally free $R$-module of rank 1 and $\dcris^+(V_{R}^\ast)^{\varphi=\alpha}=\ddR^+(V_{R}^\ast)$, where $V_R$ denotes the restriction of $V_{D}$ on $M(R)$. A crucial observation is that this property forces $\dcris^+(V_{D}^\ast)=\ddR^+(V_{D}^\ast)$; the details are presented in \S5. 

In \S6, we first apply the flat base change property of de Rham periods \cite{Bel13} to show that $\ddR^+(V^\ast_D)$ is nonzero. This ensures that the specialization of $\ddR^+(V^\ast_D)$ at the puncture gives rise to the desired crystalline periods. Finally, it is easy to deduce from $\alpha(0)\neq 0$ that the compositions
\[
D^\ast\ra\mathcal{C}_{p,N}\hookrightarrow X_p\times \mathbb{G}_m\times\prod_{\text{prime}\hspace{0.5mm}q |N} \mathbb{A}^1
\ra \mathbb{A}^1
\] 
for all prime factors $q$ of $N$ extend to $D$, concluding the proof of Theorem \ref{main}.

\subsection*{Notations}
Let $p$ be a prime number. We fix a positive integer $N$, which is coprime to $p$, to be the tame level. Choose a compatible system of primitive $p$-power roots of unity $(\zeta_{p^n})_{n\geq 0}$. Namely, each $\zeta_{p^n}$ is a primitive $p^n$-th root of unity and $\zeta_{p^{n+1}}^p=\zeta_{p^n}$. Let $\Qp(\zeta_{p^{\infty}})=\bigcup_{n\geq 1}\Qp(\zeta_{p^n})$. Let $G_{\Qp}=\Gal(\overline{\mathbb{Q}}_p/\Qp)$ and $\Gamma=\Gamma_{\Qp}=\Gal(\Qp(\zeta_{p^{\infty}})/\Qp)$.
Let $\Sigma$ be the finite set of places of $\mathbb{Q}$ consisting of the infinite place and the places dividing $pN$, and let $G_{\mathbb{Q},\Sigma}$ be the absolute Galois group of the maximal extension of $\mathbb{Q}$ which is unramified outside the places of $\Sigma$. 

For a topological group $G$ and a rigid analytic space $X$ over $\Qp$, by a \emph{family of $p$-adic representations} of $G$ of dimension $d$ on $X$ we mean a locally free coherent $\mathcal{O}_X$-module $V_X$ of rank $d$ equipped with a continuous $\mathcal{O}_X$-linear $G$-action, and we denote its dual by $V_X^\ast$. When $X=M(S)$ is an affnioid space over $\Qp$, we also call a family of $p$-adic representations of $G$ on $X$ an \textit{$S$-linear representation} of $G$. If $M(R)\subset M(S)$ is an affinoid subdomain and $V_S$ is a family of representation on $M(S)$, we write $V_R$ for the base change $V_S\otimes_SR$. Finally, for every $x\in M(S)$, we write $V_x$ to denote the specialization $V_S\otimes_S k(x)$.


\subsection*{Acknowledgements}
The authors would like to thank Rebecca Bellovin, Kevin Buzzard, Payman Kassaei, Kiran Kedlaya, Mark Kisin, and Liang Xiao for useful comments on earlier drafts of this paper. 

\section{The eigencurve $\mathcal{C}_{p,N}$}

For a classical eigenform $f$, let $\rho_f$ denote the two dimensional $p$-adic representation of $G_{\mathbb{Q}}$ associated to $f$. Here, we follow the convention so that the trace of the \emph{arithmetic Frobenius} $\mathrm{Frob}_l$ equals the $T_l$-eigenvalue of $f$, for each $l\nmid pN$. If $f$ is of level $\Gamma_1(Np^m)$ for some $m\geq 1$, let $a_q(f)$ denote its $U_q$-eigenvalue for any prime divisor $q$ of $pN$.  Recall that by a \emph{$p$-modular residual representation of tame level $N$}, we mean a two dimensional $G_{\mathbb{Q}}$-representation $\overline{V}$ over a finite field of characteristic $p$, which is isomorphic to the mod $p$ reduction $\overline{\rho}_f$ of $\rho_f$ for some classical eigenform $f$ of level $\Gamma_1(Np^m)$ with $m\geq1$. Notice that $\overline{\rho}_f$ is well-defined up to semi-simplification. Since the modular curve $Y_1(Np^m)$ has good reductions at the places outside $\Sigma$, we have that the $G_{\mathbb{Q}}$-action on $\rho_f$ factors through $G_{\mathbb{Q},\Sigma}$. Hence the $G_{\mathbb{Q}}$-action on $\overline{V}$ factors through $G_{\mathbb{Q},\Sigma}$ as well. Let $R_{\overline{V}}$ be the universal deformation ring of the $G_{\mathbb{Q},\Sigma}$-pseudo-representation associated to $\overline{V}$. It is a quotient of a power series ring in finitely many variables over $\mathbb{Z}_p$. So one can attach a rigid analytic space $X_{\overline{V}}$ to $R_{\overline{V}}[1/p]$ following the construction of Berthelot (cf. \cite[\S7]{dJ95}). 

For a classical eigenform $f$ of level $\Gamma_1(Np^m)$ for some $m\geq1$, if $\overline{\rho}_f$ is isomorphic to $\overline{V}$ up to semi-simplification, then $\rho_f$ naturally gives rise to a point of $X_{\overline{V}}(E)$, where $E$ is the field of coefficients of $\rho_f$. Moreover, if the $U_p$-eigenvalue $a_p(f)$ is nonzero, we may attach to $f$ a point 
\[
x_f=(\rho_f, a_p(f)^{-1}, \prod_{q | N} a_q(f))
\]
of $(X_{\overline{V}}\times \mathbb{G}_m\times\prod_{q | N} \mathbb{A}^1)(E)$; we call $x_f$ a \emph{modular} point. 

Let $X_p=\coprod X_{\overline{V}}$ where $\overline{V}$ runs through all isomorphism classes of $p$-modular tame level $N$ residual representations\footnote{There should be only finitely many such $\overline{V}$. For $p>2$,  this is \cite[Proposition 5.1.1]{CM98}. The case $p=2$ should hold by adapting the argument of \cite[Proposition 5.1.1]{CM98}. However, we can not find a written reference for this result at this moment.}  It turns out  that these modular points can be interpolated to a rigid analytic curve, the so called eigencurve, $\mathcal{C}_{p,N}\subset X_{p}\times \mathbb{G}_m\times\prod_{q | N} \mathbb{A}^1$, and that the points of $\mathcal{C}_{p,N}$ correspond bijectively to tame level $N$ normalized overconvergent eigenforms which are of finite slope. The case of $p>2$ and $N=1$ is established in \cite{CM98}. The general case follows from a similar argument  which is carried out in \cite{KL03}\footnote{The results in \cite{KL03} assumes $g>1$. But the arguments work equally well in the case $g=1$.}.

We have a rank 2 pseudo-representation of $G_{\mathbb{Q},\Sigma}$ on the eigencurve $\mathcal{C}_{p,N}$ by pulling back the universal pseudo-representations on $X_p$. Let $\mathcal{\tilde{C}}_{p,N}$ be the normalization of $\mathcal{C}_{p,N}$. It follows from \cite[Theorem 5.1.2]{CM98} that any rank 2 pseudo-representation of $G_{\mathbb{Q},\Sigma}$ on a smooth rigid analytic curve over $\Qp$ can be naturally converted to a family of $p$-adic representations of $G_{\mathbb{Q},\Sigma}$. Thus there exists a family of $G_{\mathbb{Q},\Sigma}$-representations $V_{\mathcal{\tilde{C}}}$ of dimension $2$ on $\mathcal{\tilde{C}}_{p,N}$ whose associated pseudo-representation is isomorphic to the pullback of the rank 2 pseudo-representation on $\mathcal{C}_{p,N}$. 

The main result of this section is that the composition
\[
u: D^\ast\stackrel{h}\to\mathcal{C}_{p,N}\ra X_p\times \mathbb{G}_m\times\prod_{q |N} \mathbb{A}^1
\ra X_p
\] 
extends to a morphism on $D$. Before proceeding, we first make the following observation.
\begin{lemma}\label{lem:extension}
Let $F\in\mathcal{O}(D^\ast)$. If $|F(x)|$ is bounded above for all $x\in D^\ast$, then $F$ extends uniquely to an element of $\mathcal{O}(D)$.
\end{lemma}
\begin{proof}
The uniqueness is obvious. After rescaling, we may suppose $|F(x)|\leq 1$ for any $x\in D^\ast$. Recall that $D$ is defined over a finite extension $L$ over $\Qp$. Let $M(L\langle T, p^nT^{-1}\rangle)\subset D^\ast$ be the closed annulus with outside radius 1 and inside radius $p^{-n}$. According to the assumption, $|F|\leq 1$ on $M(L\langle T, p^nT^{-1}\rangle)$ for all $n\geq 1$. This implies $F\in \mathcal{O}_L\langle T, p^{n}T^{-1}\rangle$ for all $n\geq1$. Hence
\[
F\in\bigcap_{n\geq1}\mathcal{O}_L\langle T, p^{n}T^{-1}\rangle=\mathcal{O}_L\langle T \rangle,
\]
yielding $F\in\mathcal{O}(D)$. \end{proof}

\begin{prop}\label{prop:extension}
The morphisms $u$ extends to a morphism of rigid analytic spaces $\tilde{u}:D\ra X_p$.
\end{prop}
\begin{proof}
Since $D^\ast$ is connected, it maps to $X_{\overline{V}}$ for some $\overline{V}$. Recall that $X_{\overline{V}}$ is the rigid analytic space associated to $R_{\overline{V}}[1/p]$.  By the construction of $X_{\overline{V}}$, it follows that $|t(x)|\leq 1$ for any $x\in X_{\overline{V}}$ and $t\in R_{\overline{V}}$.  Thus for any $y\in D^\ast$, $|u^\ast(t)(y)|=|t(u(y))|\leq 1$. By Lemma \ref{lem:extension}, $u^\ast(t)$ extends to an element of $\mathcal{O}_L\langle T\rangle$. We therefore obtain a continuous morphism $u^\ast: R_{\overline{V}}\ra \mathcal{O}_L\langle T\rangle$. The induced morphism on their rigid generic fibers gives rise to the desired extension $\tilde{u}$.
\end{proof}

We may assume that the given map $h$ in Theorem \ref{main} is dominant, otherwise the situation would become trivial.  Since $D^\ast$ is smooth, it follows that $h:D^\ast\ra \mathcal{C}_{p,N}$  factors through $\mathcal{\tilde{C}}_{p,N}$. By abuse of notation we still denote the resulting map $D^\ast\ra \mathcal{\tilde{C}}_{p,N}$ by $h$ and the composition $\mathcal{\tilde{C}}_{p,N}\ra\mathcal{C}_{p,N}\ra \mathcal{W}_N$ by $\pi$. Let $V_{D^\ast}$ denote the pullback of $V_{\mathcal{\tilde{C}}}$ via $h$. We denote by $r_{D^\ast}$ the pseudo-representation associated to $V_{D^\ast}$. It is clear that $r_{D^\ast}$ is isomorphic to the pullback of the universal pseudo-representation on $X_p$ along $u$. 
\begin{cor}\label{cor:extension}
The family of $p$-adic representations $V_{D^\ast}$ extends to $D$.
\end{cor}
\begin{proof}
Pulling back the universal pseudo-representation on $X_p$ along $\tilde{u}$, which is given by Proposition \ref{prop:extension}, we obtain a pseudo-representation $r_D$ of $G_{\mathbb{Q},\Sigma}$ on $D$ which extends $r_{D^\ast}$. Since $D$ is a smooth rigid analytic curve over $\Qp$, by \cite[Theorem 5.1.2]{CM98}, one can convert $r_D$ to a family of $p$-adic representations on $D$, yielding the desired extension.
\end{proof}
In the rest of the paper, we denote by $V_D$ the extended family of $p$-adic representations of $G_{\mathbb{Q},\Sigma}$ on $D$ given by Corollary \ref{cor:extension}. 
\section{Families of $p$-adic representations}
In this section, we give a brief review on various $p$-adic Hodge theoretic functors for families of $p$-adic representations of $G_{\Qp}$. We refer the reader to \cite{BeCo08} and \cite{KL10} for more details.


\subsection{The modules $\ddR^+(V_S)$ and $\dcris^+(V_S)$}
 Let $\bdR^+$ and $\bcris^+$ be the de Rham and crystalline period rings used in $p$-adic Hodge theory (cf. \cite{Fo82} for more details). For each $k> 0$, $\bdR^+/(t^k)$ is naturally a $\Qp$-Banach space. This gives a Fr\'echet topology on
\[
\bdR^+=\varprojlim_k \bdR^+/(t^k).
\]
Thus, for any $\Qp$-affinoid algebra $S$, it makes sense to define $S\ho_{\Qp}\bdR^+=\varprojlim S\ho_{\Qp} \bdR^+/(t^k)$.
On the other hand, the definition of $S\ho_{\Qp}\bcris^+$ is the natural one as $\bcris^+$ has a $\Qp$-Banach space structure. Moreover, for an $S$-linear representation $V_S$ of $G_{\Qp}$, we set
\[\ddR^+(V_S)=((S\ho_{\Qp}\bdR^+)\otimes_S V_S)^{G_{\Qp}},\] and \[\dcris^+(V_S)=((S\ho_{\Qp}\bcris^+)\otimes_S V_S)^{G_{\Qp}}.\]



The following proposition, which is due to Bellovin \cite{Bel13}, ensures the flat base change property of the functor $\ddR^+$. \begin{prop}\label{basechange}
If $f:S\rightarrow S'$ is a flat morphism of $\Qp$-affinoid algebras, then
\[
\ddR^+(V_S)\otimes_S S'\xrightarrow[]{\sim} \ddR^+(V_S\otimes_S S').
\]
\end{prop}
\begin{proof}
See the proof of \cite[Proposition 4.3.7]{Bel13}.
\end{proof}






\subsection{The modules $\D^{\dagger}_{\rig}(V_S)$ and $\D^+_{\dif}(V_S)$}
Recall that $\brig{,s}{,\Qp}$ and $\brig{}{,\Qp}$ are the ``Robba type" period rings in the theory of $\m$-modules introduced in \cite{Ber02}. In \cite{BeCo08}, Berger and Colmez construct the family version of $\m$-modules functor for free $S$-linear representations. This functor is later generalized to general $S$-linear representations by Kedlaya and the second author in \cite{KL10}. In the following, we first recall the definition of $\m$-modules over affinoid bases; we set
$S\ho_{\Qp}\brig{}{,\Qp}=\bigcup_{s>0}S\ho_{\Qp}\brig{,s}{,\Qp}$.

\begin{defn}
For $s>0$, a \emph{$\varphi$-module} over $S\widehat{\otimes}_{\Qp}\mathbf{B}^{\dag,s}_{\rig,\Qp}$ is a finite projective $S\widehat{\otimes}_{\Qp}\mathbf{B}^{\dag,s}_{\rig,\Qp}$-module $D_S^s$ equipped with an isomorphism $\varphi^*D_S^s\cong D_S^s\otimes_{S\widehat{\otimes}_{\Qp}\mathbf{B}^{\dag,s}_{\rig,\Qp}}S\widehat{\otimes}_{\Qp}\mathbf{B}^{\dag,ps}_{\rig,\Qp}.$ 
A \emph{$\varphi$-module} $D_S$ over $S\widehat{\otimes}_{\Qp}\mathbf{B}^{\dag}_{\rig,\Qp}$ is the base change to $S\widehat{\otimes}_{\Qp}\mathbf{B}^{\dag}_{\rig,\Qp}$ of a $\varphi$-module $D_S^s$ over $S\widehat{\otimes}_{\Qp}\mathbf{B}^{\dag,s}_{\rig,\Qp}$ for some $s>0$. 
A \emph{$\m$-module} over $S\widehat{\otimes}_{\Qp}\mathbf{B}^{\dag,s}_{\rig,\Qp}$ is a $\varphi$-module $D_S^s$ over $S\widehat{\otimes}_{\Qp}\mathbf{B}^{\dag,s}_{\rig,\Qp}$ equipped with a commuting $\mathbf{B}^{\dag,s}_{\rig,\Qp}$-semilinear and $S$-linear continuous action of $\Gamma$. A \emph{$\m$-module} $D_S$ over $S\widehat{\otimes}_{\Qp}\mathbf{B}^{\dag}_{\rig,\Qp}$ is the base change to $S\widehat{\otimes}_{\Qp}\mathbf{B}^{\dag}_{\rig,\Qp}$ of a $\m$-module $D_S^s$ over $S\widehat{\otimes}_{\Qp}\mathbf{B}^{\dag,s}_{\rig,\Qp}$ for some $s>0$. 
\end{defn}

Let $V_S$ be an $S$-linear $G_{\Qp}$-representation of rank $d$.  For any sufficiently large $s$, one can construct a $\m$-modules $\D^{\dag, s}_{\rig}(V_S)$ of rank $d$ over $S\ho_{\Qp}\bdag{,s}_{\rig,\Qp}$ such that for any $x\in M(S)$, $\D^{\dag,s}_{\rig}(V_S)\otimes_S S/\MM_x$ is naturally isomorphic to $\D_{\rig}^{\dag,s}(V_x)$. We set
\[
\D^{\dag}_\rig(V_S)=(S\ho_{\Qp}\bdag{}_{\rig,\Qp})\otimes_{S\ho_{\Qp}\bdag{,s}_{\rig,\Qp}}\D^{\dag, s}_\rig(V_S)=\bigcup_{s}\D^{\dag, s}_\rig(V_S),
\]
which is a $\m$-module of rank $d$ over $S\ho_{\Qp}\bdag{}_{\rig,\Qp}$ and specializes to $\D^{\dag}_\rig(V_x)$ for any $x\in M(S)$. Moreover, $\D^{\dag}_\rig(V_S)$ is \emph{\'etale} in the sense of \cite[Definition 6.3]{KL10} (though we do not need this fact in this paper).

Recall that for $0<s\leq r_n=p^{n-1}(p-1)$, one has the localization map
\[
\iota_n:\brig{,s}{,\Qp}\rightarrow \Qp(\zeta_{p^n})[[t]].
\]
This induces a continuous map $S\ho_{\Qp}\bdag{,s}_{\rig,\Qp}\rightarrow S\ho_{\Qp}\Qp(\zeta_{p^n})[[t]]$. Define
\[\D^{+,n}_{\dif}(V_S)=(S\ho_{\Qp}\Qp(\zeta_{p^n})[[t]])\otimes_{\iota_n,S\ho_{\Qp}\bdag{,s}_{\rig,\Qp}}\D^{\dag,s}_{\rig}(V_S).\]
It is clear that $\D^{+,n}_{\dif}(V_S)$ is a locally free $S\ho_{\Qp}\Qp(\zeta_{p^n})[[t]]$-module of rank $d$ equipped with a semilinear $\Gamma$-action. Abusing the notation, we still denote by $\iota_n$ the natural map $\iota_n:\D^{\dag,s}_{\rig}(V_S)\rightarrow \D^{+,n}_{\dif}(V_S)$. We define
\[\D^{+}_{\dif}(V_S)=\bigcup_{n}\D^{+,n}_{\dif}(V_S).\]

The following theorem generalizes Berger's comparisons between $p$-adic Hodge theory and $\m$-modules functor to $S$-linear representations.

\begin{theorem}\label{thm:comparison}\cite[Theorem 4.2.7, Theorem 4.2.8]{Bel13}
For an $S$-linear representation $V_S$ of $G_{\Qp}$, we have
$\ddR^+(V_S)=\D_{\dif}^+(V_S)^{\Gamma}$ and $\dcris^+(V_S)=\D^{\dagger}_{\rig}(V_S)^{\Gamma}$.
\end{theorem}


\section{Finite slope subspace of $D^\ast$}

Let $X$ be a reduced and separated rigid analytic space over $\Qp$, and let $V_X$ be a family of $p$-adic representations over $X$ of $G_{\Qp}$. Recall that the Sen operator $\nabla=\log\gamma/\log\chi(\gamma)$, where $\gamma\in\Gamma$ is not of finite order,  is independent of the choice of $\gamma$ and gives rise to an action of the Lie algebra of $\Gamma$ on $\D_\dif^{+}(V_X)$. The Sen polynomial is defined to be the characteristic polynomial of $\nabla$ on $\D_{\mathrm{dif}}^{+}(V_X)$, which belongs to $\mathcal{O}(X)[u]$. Moreover,  the roots of the Sen polynomial are exactly the Hodge-Tate-Sen weights of $V_X$. We refer the reader to \cite{Liu12} for more details about Sen operator and polynomial in this context. 

From now on, we assume that $V_X$ has $0$ as a Hodge-Tate-Sen weight. Therefore we may write the the Sen polynomial of $V_X$ as $uQ(u)$ for some $Q(u)\in\mathcal{O}(X)[u]$. Let $\alpha\in\mathcal{O}(X)^\times$. Recall that in \cite{Liu12}, the second author introduces the notion of finite slope subspaces of $X$ with respect to the pair $(\alpha,V_X)$, which refines the original definition of finite slope subspaces introduced by Kisin \cite{Kis03}. In the following, for any rigid analytic space $Y$ over $\Qp$ and $f\in\mathcal{O}(Y)$, we denote by $Y_f$ the complement of the vanishing locus of $f$ on $Y$; it is a Zariski open subset of $Y$.  

\begin{defn}\label{def:fs-space}
For such a triple $(X,\alpha, V_X)$, we call an analytic subspace $X_{fs}\subset X$ a \emph{finite slope subspace} of $X$ with respect to the pair $(\alpha,V_X)$ if it satisfies the following conditions.
\begin{enumerate}
\item[(1)]For every integer $j\leq0$, the subspace $(X_{fs})_{Q(j)}$ is Zariski open and dense in $X_{fs}$.
\item[(2)]For any affinoid algebra $R$ over $\Qp$ and morphism $g:M(R)\ra X$ which factors through $X_{Q(j)}$ for every integer $j\leq0$, the morphism $g$ factors through $X_{fs}$ if and only if the natural map
\begin{equation}\label{eq:cris-dR}
   \iota_{n}: (\mathrm{D}^{\dag}_{\rig}(V_R))^{\varphi=g^*(\alpha),\Gamma=1}\ra \mathrm{D}_{\dif}^{+,n}(V_R)^{\Gamma}
\end{equation}
is an isomorphism for all sufficiently large $n$.
\end{enumerate}
\end{defn}
\noindent Furthermore, \cite[Theorem 3.3.1]{Liu12} ensures that $X$ has a unique finite slope subspace $X_{fs}$ associated to the pair $(\alpha,V_X)$.

Now let $\alpha_{\mathcal{C}}\in \mathcal{O}(\mathcal{C}_{p,N})^\times$ be the $U_p$-eigenvalue. That is, for any $x\in\mathcal{C}_{p,N}$, $\alpha(x)$ is the $U_p$-eigenvalue of the overconvergent eigenform represented by $x$. Let $\alpha_{\mathcal{\tilde{C}}}\in \mathcal{O}(\mathcal{\tilde{C}}_{p,N})^\times$ and $\alpha\in\mathcal{O}(D^\ast)^\times$ be the pullbacks of $\alpha_{\mathcal{C}}$. Since the family of $p$-adic representations $V_{\mathcal{\tilde{C}}}$ has 0 as a Hodge-Tate-Sen weight, we may write the Sen polynomial of $V^\ast_{\mathcal{\tilde{C}}}$ as $T(T-\kappa_{\mathcal{\tilde{C}}})$ for some $\kappa_{\mathcal{\tilde{C}}}\in\mathcal{O}(\mathcal{\tilde{C}}_{p,N})$. For any $x\in\mathcal{\tilde{C}}_{p,N}$, we say $x$ is of \emph{integral weight} if $\kappa_{\tilde{\mathcal{C}}}(x)\in\mathbb{Z}$. In particular, the Hodge-Tate weights of $V_x$ are all integers. Let $\kappa=h^\ast(\kappa_{\mathcal{\tilde{C}}})\in\mathcal{O}(D^\ast)$. It follows that the Sen polynomial of $V^\ast_{D^\ast}$ is $T(T-\kappa)$.
\begin{prop}
The finite slope subspace $(D^{\ast})_{fs}$ of the punctured disk $D^{\ast}$ associated to $(\alpha, V^\ast_{D^{\ast}})$ is $D^{\ast}$ itself.
\end{prop}

\begin{proof}
To show the proposition, we just need to check that the triple ($D^{\ast}, \alpha, V^\ast_{D^\ast}$) satisfies the conditions (1) and (2) of Definition \ref{def:fs-space}. Since the finite slope subspace associated to the pair $(\alpha_{\mathcal{\tilde{C}}}, V^\ast_{\mathcal{\tilde{C}}})$ is $\mathcal{\tilde{C}}_{p,N}$ itself by the main results of \cite{Liu12}, we know that the triple ($\mathcal{\tilde{C}}_{p,N}, \alpha_{\mathcal{\tilde{C}}}, V^\ast_{\mathcal{\tilde{C}}}$) satisfies the conditions (1) and (2). Hence $(\mathcal{\tilde{C}}_{p,N})_{(\kappa_{\mathcal{\tilde{C}}}-j)}$ is Zariski open and  dense in $\mathcal{\tilde{C}}$ for every $j\leq0$. Since $h$ is dominant and $D^\ast$ is of dimension 1, we deduce that $D^\ast_{(\kappa-j)}=h^{-1}((\mathcal{\tilde{C}}_{p,N})_{(\kappa_{\mathcal{\tilde{C}}}-j)})$ is Zariski open and dense in $D^\ast$. Thus the triple ($D^{\ast}, \alpha, V^\ast_{D^\ast}$) satisfies the condition (1). It follows immediately that the triple ($D^{\ast}, \alpha, V^\ast_{D^\ast}$) satisfies the condition (2) as well because $D^\ast_{(\kappa-j)}$ maps to
$(\mathcal{\tilde{C}}_{p,N})_{(\kappa_{\mathcal{\tilde{C}}}-j)}$ for every $j\leq 0$.
\end{proof}

\begin{prop}\label{isomLiu}
For any affinoid subdomain $M(R)$ of $D^\ast$ and $k>\log_p|\alpha^{-1}|$, where $|\cdot|$ denotes the norm taken on $M(R)$,
the natural map
\[
(\D_\rig^\dag(V^\ast_R)))^{\varphi=\alpha,\Gamma=1}\rightarrow (\D_\dif^+(V^\ast_R)/(t^k))^\Gamma
\]
is an isomorphism. Furthermore, $(\D_\rig^\dag(V^\ast_R)))^{\varphi=\alpha,\Gamma=1}$ is a locally free $R$-module of rank $1$.
\end{prop}

\begin{proof}
Since the finite slope subspace of $D^\ast$ is itself, it follows immediately from \cite[Theorem 3.3.3]{Liu12} that the given map is an isomorphism. Note that $M(R)$ is smooth of dimension 1. We deduce that the $R$-module $(\D_\dif^+(V_R)/(t^k))^\Gamma$ is locally free because it is finite and torsion free. Moreover, let $x\in M(R)$ be a point with non-integral weight; that is, $h(x)\in\mathcal{\tilde{C}}_{p,N}$ is not of integral weight. Then by \cite[Corollary 1.5.7]{Liu12}, the natural map
\[
(\D_\dif^+(V^\ast_R)/(t^k))^\Gamma\otimes_{R}k(x)\ra (\D_\dif^+(V^\ast_x)/(t^k))^\Gamma
\]
is an isomorphism. The right hand side is of $k(x)$-dimension 1 by \cite[Proposition 4.1.5(4)]{Liu12}. Thus $(\D_\dif^+(V^\ast_R)/(t^k))^\Gamma$ is locally free of rank 1 around $x$. 

By the main results of \cite{Buz07}, the projection from any irreducible component of $\mathcal{C}_{p,N}$ to the weight space is locally in-the-domain finite flat. We therefore deduce that the composition $\pi\circ h$ is dominant.  This implies that the restriction of $\pi\circ h$ on any irreducible component of $M(R)$ is dominant. We therefore deduce that the subset of points with non-integral weights is Zariski dense in $M(R)$. This yields that $(\D_\dif^+(V^\ast_R)/(t^k))^\Gamma$ is a locally free $R$-module of rank 1, and so is $(\D_\rig^\dag(V^\ast_R)))^{\varphi=\alpha,\Gamma=1}$.
\end{proof}

\begin{cor}\label{cor:comparison}
For any affinoid subdomain $M(R)$ of $D^\ast$, the natural map
\[
\dcris^+(V^\ast_R)^{\varphi=\alpha}\ra\ddR^+(V^\ast_R)
\]
is an isomorphism. Furthermore, they are locally free $R$-modules of rank 1.
\end{cor}
\begin{proof}
The previous proposition implies that for sufficiently large $k$, the natural map
\[
(\D_\rig^\dag(V^\ast_R)))^{\varphi=\alpha,\Gamma=1}\rightarrow (\D_\dif^+(V^\ast_R)/(t^k))^\Gamma
\]
is an isomorphism, yielding that the natural map
\[
(\D_\rig^\dag(V^\ast_R)))^{\varphi=\alpha,\Gamma=1}\rightarrow \D_\dif^+(V^\ast_R)^\Gamma=\displaystyle{\varprojlim_{k}}(\D_\dif^+(V^\ast_R)/(t^k))^\Gamma
\]
is an isomorphism. We conclude by applying Theorem \ref{thm:comparison}.
\end{proof}


\section{De Rham periods vs. crystalline periods}
The goal of this section is to show $\ddR^+(V^\ast_D)=\dcris^+(V^\ast_D)$. The upshot is to show if a de Rham period of $V_D^\ast$ is crystalline over a closed annulus of outside radius $1$, then it is crystalline on all of $D$.  
\begin{defn} Let $A$ be a Banach algebra over $\Qp$.
\begin{enumerate}
\item[(i)] For any $n\geq 0$, define the Banach algebra $A\langle p^{-n}T\rangle$ to be the ring of  formal power series $\sum_{i\in\mathbb{N}} a_iT^i$ with $a_i\in A$ and such that $|a_i|p^{-ni}\rightarrow 0$ as $i\rightarrow \infty$. It is equipped with a Banach norm $|\sum_{i\in\mathbb{N}}a_iT^i|=\sup|a_i|p^{-ni}$.
\item[(ii)] For any $n'>n\geq 0$, define the Banach algebra $A\langle p^{-n}T, p^{n'}T^{-1}\rangle$ to be the ring of Laurent series $\sum_{i\in \ZZ} a_iT^i$ with $a_i\in A$ and such that $|a_i|p^{-ni}\rightarrow 0$ as $i\rightarrow\infty$ and $|a_i|p^{-n'i}\rightarrow 0$ as $i\rightarrow -\infty$. It is equipped with a Banach norm $|\sum_{i\in \ZZ}a_iT^i|=\max_{i\in\mathbb{Z}}\{\sup|a_i|p^{-ni}, \sup|a_i|p^{-n'i}\}$.
\end{enumerate}
\end{defn}

In the rest of the paper, let $S=L\langle T\rangle=\mathcal{O}(D)$. For any $n\geq 0$ (resp. $n'>n\geq 0$), let $S_n=L\langle p^{-n}T\rangle$ (resp. $S_{n,n'}=L\langle p^{-n}T, p^{n'}T^{-1}\rangle$). Let $V_{n}$ (resp. $V_{n,n'}$) be the restriction of $V_D$ on $M(S_n)$ (resp. $M(S_{n,n'})$). Let $A_L=A\otimes_{\Qp}L$. Using the fact that $\{p^{-ni}T^i\}_{i\in\mathbb{N}}$ and $\{p^{-ni}T^i, p^{n'(i+1)}T^{-i-1}\}_{i\in\mathbb{N}}$ form $L$-orthonormal bases of $S_n$ and $S_{n,n'}$ respectively, we deduce the following lemma.

\begin{lemma}\label{banach}
Let $A$ be a $\Qp$-Banach algebra. For any $n\geq 0$, we have natural identification of Banach algebras
\[\eta_{n, A}: S_n\ho_{\Qp}A\overset{\sim}{\longrightarrow}A_L\langle p^{-n}T\rangle. \]
Similarly, for any $n'>n\geq 0$, we have natural identification of Banach algebras
\[\eta_{n,n',A}:S_{n,n'}\ho_{\Qp}A\overset{\sim}{\longrightarrow}A_L\langle p^{-n}T, p^{n'}T^{-1}\rangle.\]
\end{lemma}

\begin{defn}\label{def:complete-tensor-frechet}
Let $A=\displaystyle{\varprojlim_{j\in J} A_j}$ be a Fr\'echet algebra where $A_j$'s are $\Qp$-Banach algebras.

\begin{enumerate}
\item[(i)]Define the Fr\'echet algebra $A\langle p^{-n}T\rangle$ to be the inverse limit of Banach algebras $A_j\langle p^{-n}T\rangle$.

\item[(ii)]For any $n'>n\geq 0$, define the Fr\'echet algebra $A\langle p^{-n}T, p^{n'}T^{-1}\rangle$ to be the inverse limit of Banach algebras $A_j\langle p^{-n}T, p^{n'}T^{-1}\rangle$.
\end{enumerate}
\end{defn}

Note that the natural inclusions $A_j\langle p^{-n}T\rangle\hookrightarrow A_j[[T]]$
induces an injective map
\[
A\langle p^{-n}T\rangle=\displaystyle{\varprojlim_{j\in J}}
A_j\langle p^{-n}T\rangle\hookrightarrow \displaystyle{\varprojlim_{j\in J}}
A_j[[T]]=A[[T]].
\]
Thus one may naturally identify $A\langle p^{-n}T\rangle$ as a subring of $A[[T]]$.
Similarly, one can naturally identify $A\langle p^{-n}T, p^{n'}T^{-1}\rangle$ as a subset of $A[[T, T^{-1}]]$, which is the set of Laurent series with coefficients in $A$; note that $A[[T,T^{-1}]]$ is not a ring!

\begin{lemma}\label{frechet}
Keep notations as in Definition \ref{def:complete-tensor-frechet}. For any $n\geq 0$, we have natural identification of Fr\'echet algebras
\[
\eta_{n, A}: S_n\ho_{\Qp}A\overset{\sim}{\longrightarrow}A_L\langle p^{-n}T\rangle.
\]
Similarly, for any $n'>n\geq 0$, we have natural identification of Fr\'echet algebras
\[\eta_{n,n',A}:S_{n,n'}\ho_{\Qp}A\overset{\sim}{\longrightarrow}A_L\langle p^{-n}T, p^{n'}T^{-1}\rangle.
\]
\end{lemma}

\begin{proof}
Apply the previous lemma to the $\Qp$-Banach algebras $A_j$ and take inverse limits.
\end{proof}

In particular, Lemma~\ref{banach} applies to $A=\bcris^+$ and Lemma~\ref{frechet} applies to $A=\bdR^+=\varprojlim \bdR^+/(t^i)$.

\begin{lemma}\label{lastlemma}
\begin{enumerate}
\item[(i)] For any $n\geq 0$, the continuous map $\bcris^+\rightarrow \bdR^+$ induces a natural inclusion $(\bcris^+\otimes_{\Qp}L)\langle p^{-n}T\rangle \hookrightarrow (\bdR^+\otimes_{\Qp}L)\langle p^{-n}T\rangle$.
\item[(ii)] For any $n'>n\geq 0$, the continuous map $\bcris^+\rightarrow \bdR^+$ induces a natural inclusion 
\[
(\bcris^+\otimes_{\Qp}L)\langle p^{-n}T, p^{n'}T^{-1}\rangle \hookrightarrow (\bdR^+\otimes_{\Qp}L)\langle p^{-n}T, p^{n'}T^{-1}\rangle.
\]
\end{enumerate}

\end{lemma}

\begin{proof}
By the commutative diagram

\[\xymatrix@C=90pt@R=50pt{
(\bcris^+\otimes_{\Qp}L)\langle p^{-n}T\rangle\ar[r]^{}\ar[d]^{} &  (\bdR^+\otimes_{\Qp}L)\langle p^{-n}T \rangle\ar[d]^{}\\
(\bcris^+\otimes_{\Qp}L)[[T]]\ar[r]^{} &  (\bdR^+\otimes_{\Qp}L)[[T]],
}
\]
we see that the composition $(\bcris^+\otimes_{\Qp}L)\langle p^{-n}T\rangle\ra (\bdR^+\otimes_{\Qp}L)\langle p^{-n}T\rangle\ra (\bdR^+\otimes_{\Qp}L)[[T]]$ is injective. Hence the natural map $(\bcris^+\otimes_{\Qp}L)\langle p^{-n}T\rangle\ra (\bdR^+\otimes_{\Qp}L)\langle p^{-n}T\rangle$ is injective. The proof of (ii) is similar.
\end{proof}
As a consequence of Lemma \ref{lastlemma}, we may naturally identify $S_n\ho_{\Qp}\bcris^+$ (resp. $S_{n,n'}\ho_{\Qp}\bcris^+$) as a subring of $S_n\ho_{\Qp}\bdR^+$ (resp. $S_{n,n'}\ho_{\Qp}\bdR^+$).
\begin{lemma} \label{keylemma}
For any $x\in S_n\ho_{\Qp}\bdR^+$, if its image in $S_{n, n'}\ho_{\Qp} \bdR^+$
belongs to  $S_{n, n'}\ho_{\Qp}\bcris^+$, then $x\in S_n\ho_{\Qp}\bcris^+$.
\end{lemma}
\begin{proof}
By the previous lemmas, we may regard all the rings involved as subsets of $(\bdR^+\otimes_{\Qp}L)[[T,T^{-1}]]$. It follows from the assumption that $x$ belongs to
\begin{equation*}
\begin{split}
&(\bdR^+\otimes_{\Qp}L)\langle p^{-n}T\rangle\cap(\bcris^+\otimes_{\Qp}L)\langle p^{-n}T, p^{n'}T^{-1}\rangle\\
\subseteq&(\bdR^+\otimes_{\Qp}L)[[T]]\cap(\bcris^+\otimes_{\Qp}L)\langle p^{-n}T, p^{n'}T^{-1}\rangle\\
=&(\bcris^+\otimes_{\Qp}L)\langle p^{-n}T\rangle,
\end{split}
\end{equation*}
yielding the desired result.
\end{proof}

\begin{cor}\label{cor:dR-vs-crys}
$\dcris^+(V^\ast_D)=\ddR^+(V^\ast_D)$.
\end{cor}
\begin{proof}
Pick an  $S$-basis $e_1, e_2$ of $V_D^\ast$. By Corollary \ref{cor:comparison}, we get $\dcris^+(V^\ast_{0,1})=\ddR^+(V^\ast_{0,1})$. This implies that for any $a=a_1e_1+a_2e_2\in \ddR^+(V^\ast_D)$, the images of $a_1, a_2$ in $S_{0, 1}\ho_{\Qp} \bdR^+$
belong to  $S_{0, 1}\ho_{\Qp}\bcris^+$.
Applying the previous lemma, we thus obtain that
$a_1, a_2$ belong to $S\ho_{\Qp}\bcris^+$; this yields $\ddR^+(V^\ast_D)\subseteq\dcris^+(V^\ast_D)$. Hence $\ddR^+(V^\ast_D)=\dcris^+(V^\ast_D)$.
\end{proof}


\section{Proof of Theorem \ref{main}}

Applying Proposition~\ref{basechange} to $S\rightarrow S_{0,1}$, we obtain isomorphism
$\ddR^+(V^\ast_D)\otimes_{S}S_{0,1}\xrightarrow[]{\sim}\ddR^+(V^\ast_{0,1})$.
By Corollary \ref{cor:comparison}, $\ddR^+(V^\ast_{0,1})$ is a locally free $S_{0,1}$-module of rank 1. In particular, this implies $\ddR^+(V^\ast_D)\neq 0$. Now pick a nonzero element $e\in \ddR^+(V^\ast_D)$. By dividing a suitable power of $T$, we may assume that the specialization $e_{0}$ of $e$ at the puncture $0$ is nonzero. Note that $e\in\dcris^+(V^\ast_D)$ by Corollary \ref{cor:dR-vs-crys}. Moreover, the image of $e$ in $\dcris^+(V^\ast_{0,1})$ belongs to $\dcris^+(V^\ast_{0,1})^{\varphi=\alpha}$ by Corollary \ref{cor:comparison}. That is, $\varphi(e)=\alpha e$ on $M(S_{0,1})$. Following the construction of $\mathcal{C}_{p,N}$ given in \cite{Buz07}, we see that the Fredholm determinant of the compact operator $U_p$ has coefficients in $\Lambda_N=\mathbb{Z}_p[[(\mathbb{Z}/pN)^\times]]$. This implies that the norms of the $U_p$-eigenvalues of overconvergent eigenforms are always less than or equal to 1. Therefore, $\alpha$ extends to an element in $\mathcal{O}(D)$ by Lemma \ref{lem:extension}. Since $M(S_{0,1})$ is Zariski dense in $D$, we must have $\varphi(e)=\alpha e$ on the entire disk. In particular, $\varphi(e_0)=\alpha(0) e_0$. Using that fact that $\varphi$ is injective on $\dcris^+(V^\ast_0)$, we conclude that $\alpha(0)\neq 0$.
Thus $\alpha\in \mathcal{O}(D)^\times$. 

Similarly, for any prime factor $q$ of $N$, the Fredholm determinant of the compact operator $U_qU_p$ has coefficients in $\Lambda_N$. Thus the norms of the $U_qU_p$-eigenvalues of  overconvergent eigenforms are less than or equal to 1. Using Lemma \ref{lem:extension},  we deduce that the composition  
$u_q: D^\ast\ra\mathcal{C}_{p,N}\ra\mathbb{A}^1$,
where the last map sends $x$ to the product of its $U_q$ and $U_p$-eigenvalues, extends to a morphism $\tilde{u}_q$ on $D$.

Now we construct a morphism  $\tilde{h}:D\ra X_p\times \mathbb{G}_m\times\prod_{q |N} \mathbb{A}^1$ of rigid analytic spaces by sending $x$ to $(\tilde{u}(x), \alpha(x)^{-1}, \prod_{q |N}\alpha(x)^{-1}\tilde{u}_q(x))$, where $\tilde{u}$ is given by Proposition \ref{prop:extension}. It is clear that $\tilde{h}|_{D^\ast}=h$. Since $\mathcal{C}_{p,N}$ is an analytic subspace of $X_p\times \mathbb{G}_m\times\prod_{q |N} \mathbb{A}^1$, $\tilde{h}^{-1}(\mathcal{C}_{p,N})$ is an analytic subspace of $D$ containing $D^\ast$. This forces $\tilde{h}^{-1}(\mathcal{C}_{p,N})=D$, confirming that $\tilde{h}$ is the desired extension of $h$.


\end{document}